\documentclass[11pt, a4paper]{amsart}
\usepackage{a4wide}
\usepackage{hyperref}
\usepackage{amssymb}
\usepackage{mathrsfs}
\usepackage{xspace}
\usepackage{enumitem}
\usepackage[all]{xy}
\usepackage[usenames,dvipsnames]{color}

\theoremstyle{plain}
\newtheorem{prop}{Proposition}
\newtheorem{propdef}[prop]{Proposition/Definition}
\newtheorem*{prop*}{Proposition}
\newtheorem{thm}[prop]{Theorem}
\newtheorem*{thm*}{Theorem}
\newtheorem{cor}[prop]{Corollary}
\newtheorem*{cor*}{Corollary}
\newtheorem{lem}[prop]{Lemma}

\newtheorem*{convention*}{Convention}
\newtheorem*{question*}{Question}
\theoremstyle{definition}
\newtheorem*{defn*}{Definition}
\newtheorem{defn}[prop]{Definition}
\newtheorem{rem}[prop]{Remark}
\newtheorem*{rem*}{Remark}

\newtheorem*{scholium*}{Scholium}
\newtheorem*{ack*}{Acknowledgements}

\newtheorem*{example*}{Example}
\numberwithin{equation}{section}

\newcommand{\fhi}{\varphi}
\newcommand{\ro}{\varrho}

\newcommand{\CC}{\mathbf{C}}

\newcommand{\NN}{\mathbf{N}}

\newcommand{\RR}{\mathbf{R}}

\newcommand{\sF}{\mathscr{F}}

\newcommand{\sL}{\mathscr{L}}

\newcommand{\ru}{\mathrm{C}_\mathrm{b}^\mathrm{ru}}
\newcommand{\lu}{\mathrm{C}_\mathrm{b}^\mathrm{lu}}
\newcommand{\cb}{\mathrm{C}_\mathrm{b}}
\newcommand{\cc}{\mathrm{C}}
\newcommand{\lw}{L^\infty_{\mathrm{w}*}}

\newcommand{\inv}{^{-1}}
\newcommand{\se}{\subseteq}
\newcommand{\lra}{\longrightarrow}

\newcommand{\acts}{\curvearrowright}
\newcommand{\one}{\boldsymbol{1}}
\newcommand{\id}{\,\mathrm{id}}
\newcommand{\centra}{\mathscr{Z}}

\DeclareMathOperator{\ulim}{ulim}
\DeclareMathOperator{\Ramen}{Ramen}

\newcommand{\cx}{\mathscr{X}}

\begin{document}
\title{Relative amenability}
\author[P.-E. Caprace]{Pierre-Emmanuel Caprace*}
\address{UCL -- Math, Chemin du Cyclotron 2, 1348 Louvain-la-Neuve, Belgium}
\email{pe.caprace@uclouvain.be}
\thanks{* F.R.S.-FNRS research associate. Supported in part by the ERC}
\author[N. Monod]{Nicolas Monod$^\ddagger$}
\address{EPFL, 1015 Lausanne, Switzerland}
\email{nicolas.monod@epfl.ch}
\thanks{$^\ddagger$Supported in part by the Swiss National Science Foundation and the ERC}
\date{October 2013}
\keywords{Amenability, subgroups, Chabauty topology, approximate identity}
\begin{abstract}
We introduce a relative fixed point property for subgroups of a locally compact group, which we call \emph{relative amenability}. It is a priori weaker than amenability. We establish equivalent conditions, related among others to a problem studied by Reiter in~1968. We record a solution to Reiter's problem.

\smallskip
We study the class~$\cx$ of groups in which relative amenability is equivalent to amenability for all closed subgroups; we prove that $\cx$ contains all familiar groups. Actually, no group is known to lie outside~$\cx$.

\smallskip
Since relative amenability is closed under Chabauty limits, it follows that any Chabauty limit of amenable subgroups remains amenable if the ambient group belongs to the vast class~$\cx$.
\end{abstract}
\maketitle

\section{Introduction}
Let $G$ be a locally compact group. Recall that a \textbf{convex compact $G$-space} is a convex compact subset of any (Hausdorff) locally convex topological vector space endowed with a continuous affine representation of $G$ preserving this set. The group $G$ is called \textbf{amenable} if it fixes a point in every non-empty convex compact $G$-space. In this paper, we do not assume $G$ to be amenable, but focus rather on the  property of amenability among closed subgroups of $G$.  To this end, we introduce the following \emph{relative} fixed point property. 

\begin{defn*}
A closed subgroup $H < G$ is called \textbf{relatively amenable} (or \textbf{amenable relative to $G$}) if 
$H$ fixes a point in every non-empty convex compact $G$-space.
\end{defn*}

Every amenable subgroup of $G$ is thus relatively amenable. Perhaps it might come as a surprise that relative amenability is formally weaker than amenability. The purpose of this paper is to elucidate the relations between these two notions.  

\medskip
One of our initial motivations to consider relative amenability came from the following question about the space $\mathscr S(G)$ of closed subgroups of $G$ endowed with the compact topology defined by Chabauty~\cite{Chabauty}. 

\begin{question*}
Is the set of amenable subgroups closed in $\mathscr S(G)$? In other words, is amenability a closed property with respect to the Chabauty topology? 
\end{question*}

This problem seems to be open; the only results for general groups that we are aware of are the almost trivial cases where the limit is either open or contains the subgroups converging to it (see Section~\ref{sec:cors}). By contrast, it is straightforward to check that relative amenability is a Chabauty-closed property (Lemma~\ref{lem:chabauty}). Let us clarify when the two notions coincide:

\begin{propdef}\label{prop:clarify}
Given a locally compact group $G$, the following properties are equivalent. 

\begin{enumerate}[label=(\roman*)]
\item Every relatively amenable subgroup is amenable. 

\item There exists a non-empty convex compact $G$-space such that the stabiliser of every point is amenable. 
\end{enumerate}
We denote by $\cx$ the class of locally compact groups satisfying these equivalent conditions.
\end{propdef}

The point of our first theorem is that the class~$\cx$ is very large indeed. For instance, it is almost immediate
that it contains any group \textbf{amenable at infinity}, i.e.\ admitting an amenable continuous action on some compact space~\cite{Anantharaman02}. This is the case e.g.\ for all connected groups~\cite[3.3]{Anantharaman02}, all algebraic groups over local fields, and all automorphism groups of (possibly non-Euclidean) locally finite buildings~\cite{Lecureux}. For discrete groups, it is equivalent to \textbf{exactness}~\cite{Anantharaman02, Ozawa_exact}. The only groups asserted to fail exactness are the so-called Gromov monsters~\cite{GromovRANDOM} (or groups containing them).

\begin{thm}[The class~$\cx$ is very large]\label{thm:closure}
\
\begin{enumerate}[label=(\alph*)]
\item $\cx$ contains all discrete groups.\label{pt:class:discrete}
\item $\cx$ contains all groups amenable at infinity.\label{pt:class:mai}
\item $\cx$ is closed under taking closed subgroups.\label{pt:class:sub}
\item $\cx$ is closed under taking (finite) direct products.\label{pt:class:direct}
\item $\cx$ is closed under taking adelic products.\label{pt:class:adele}
\item $\cx$ is closed under taking directed unions of open subgroups.\label{pt:class:union}
\end{enumerate}
Let $N\lhd G$ be a closed normal subgroup of a locally compact group $G$.
\begin{enumerate}[resume*]
\item If $N$ is amenable, then $G\in\cx \Longleftrightarrow G/N\in\cx$.\label{pt:class:ameneq}
\item If $N$ is connected, then $G\in\cx \Longleftrightarrow G/N\in\cx$.\label{pt:class:conneq}
\item If $N$ is open, then $G\in\cx \Longleftrightarrow N\in\cx$.\label{pt:class:open}
\item If $N$ is discrete and $G/N\in\cx$, then $G\in\cx$.\label{pt:class:td:discrete}
\item If $N$ is amenable at infinity and $G/N\in\cx$, then $G\in\cx$.\label{pt:class:td:mai}
\end{enumerate}
\end{thm}

We do not know of any group outside~$\cx$; in fact, we don't even know a group for which we could conjecture that it lies outside~$\cx$. Any example can be assumed totally disconnected by~\ref{pt:class:conneq} and compactly generated by~\ref{pt:class:union}. See also Section~\ref{sec:exo} for further discussion.

\medskip
It can happen that for a specific closed subgroup $H < G$ amenability follows from relative amenability even without knowing that $G$ belongs to $\cx$. This is for example the case if one assumes $H$ to be open, or normal, or with open normaliser, as a consequence of the following result. Derighetti~\cite{Derighetti78} considered the following condition: the trivial representation~$\one_H$ is weakly contained in the restriction to $H$ of the quasi-regular representation on $L^2(G/H)$. Recall that the latter is the induced representation $\mathrm{Ind}_H^G \one_H$ and notice that $L^2(G/H)$ actually contains~$\one_H$ if $H$ has open normaliser in $G$. An example without Derighetti's condition is $\mathrm{SL}_2(\RR)$ in $\mathrm{SL}_2(\CC)$, see~\cite{Derighetti78}.

\begin{prop}\label{prop:derigh}
Let $G$ be a locally compact group and $H < G$ be a closed subgroup with \textbf{Derighetti's condition}: $\one_H \prec(\mathrm{Ind}_H^G \one_H) |_H$.

If $H$ is relatively amenable, then it is amenable. 
\end{prop}

Coming back to the question of limits of amenable groups, we deduce a positive answer in a large number of cases thanks to Theorem~\ref{thm:closure}.

\begin{cor}[Limits of amenable subgroups]\label{cor:limit}
Let $G$ be a locally compact group and $H<G$ a closed subgroup which is a Chabauty limit of amenable closed subgroups of $G$. Then $H$ is amenable provided that at least one of the following conditions holds:

\begin{enumerate}[label=(\roman*)]
\item $G$ belongs to $\cx$;

\item $H<G$ satisfies Derighetti's condition (e.g.\ it is open or normal).
\end{enumerate}
\end{cor}

In particular, if the answer to the above question is negative in general, then any counter-example would provide an example of a group not amenable at infinity, which   would necessarily be different from the only known ones since it must be non-discrete by Theorem~\ref{thm:closure}.

\bigskip

The proofs of the results presented thus far rely on detailed comparisons of relative amenability and amenability of closed subgroups from various functional analytic viewpoints. In order to discuss those, let us first recall that there is a great number of well-known characterisations of amenability of the locally compact group $G$, for instance by the existence of an invariant mean on $L^\infty(G)$, on $\cb(G)$ or on $\ru(G)$, where the latter denotes the bounded right uniformly continuous functions. Turning to a closed subgroup $H<G$, we first describe relative versions of such characterisations; see Section~\ref{sec:notation} for details on the definitions and notation.

\begin{thm}[Relative amenability]\label{thm:equiv}
Let $G$ be a locally compact group and $H<G$ a closed subgroup. The following are equivalent:
\begin{enumerate}[label=(\roman*)]
\item Every non-empty convex compact $G$-space admits an $H$-fixed point.\label{pt:equiv:fixed}
\item The kernel $J^1(G,H)$ of the map  $L^1(G)\to L^1(G/H)$ has a bounded right approximate identity in $L^1_0(G)$.\label{pt:equiv:approx}
\item There is a left $H$-invariant mean on $\ru(G)$.\label{pt:equiv:mean}
\item There is a $G$-equivariant continuous linear map $\alpha\colon L^\infty(G)\to L^\infty(G/H)$ with $\alpha(\one_G)=\one_{G/H}$.\label{pt:equiv:op}
\item There is a map as in~\ref{pt:equiv:op} which is positive and of norm one.\label{pt:equiv:opno}
\end{enumerate}
Moreover, $G$-equivariance can be replaced by $L^1(G)$-equivariance in~\ref{pt:equiv:op} and~\ref{pt:equiv:opno}.
\end{thm}

Now comes a point-by-point comparison of these criteria with the case of genuine (i.e. non-relative) amenability of the subgroup.

\begin{thm}[The classical picture]\label{thm:classic}
Let $G$ be a locally compact group and $H<G$ a closed subgroup. The following are equivalent to the amenability of $H$:
\begin{enumerate}[label=(\roman*)]
\item Every convex compact $G$-space admits an $H$-fixed point in the closed convex hull of any $H$-orbit.\label{pt:classic:fixed}
\item $J^1(G,H)$ has a bounded right approximate identity in $L^1_0(H)$.\label{pt:classic:approx}
\item There is a left $H$-invariant mean on $\cb(G)$, or equivalently on $L^\infty(G)$.\label{pt:classic:mean}
\item There is a $G$-equivariant continuous linear map $\alpha\colon L^\infty(G)\to L^\infty(G/H)$ which is the identity on $L^\infty(G/H)$.\label{pt:classic:op}
\item There is a $G$-equivariant conditional expectation $\alpha\colon L^\infty(G)\to L^\infty(G/H)$.\label{pt:classic:opno}
%
\end{enumerate}
Moreover, $G$-equivariance can be replaced by $L^1(G)$-equivariance in~\ref{pt:classic:op} and~\ref{pt:classic:opno}.
\end{thm}

\noindent
(Recall that a \textbf{conditional expectation} is by definition a positive norm one $L^\infty(G/H)$-linear map.)

\smallskip
The equivalence of~\ref{pt:classic:approx} with amenability in Theorem~\ref{thm:classic} is due to Derighetti~\cite{Derighetti78}. The criterion~\ref{pt:classic:mean} is very classical. As for~\ref{pt:classic:opno}, it was first proved in~\cite[4.4.5]{Anantharaman03}. This criterion is related to~\cite[Thm.~A]{Adams-Elliott-Giordano} (see Section~\ref{sec:measure} regarding an apparent gap in the latter).

\bigskip

The characterisations in terms of $J^1(G, H)$ in Theorems~\ref{thm:equiv} and~\ref{thm:classic} are relevant to the following question due to Reiter (\cite{Reiter68CRAS}, \cite{Reiter68} and~\cite[\S12 Rem.~1]{Reiter71}): \itshape Given a closed subgroup $H < G$, is it true that   $H$ is amenable if and only if $J^1(G,H)$ has a bounded right approximate identity?\upshape

This was known when $H$ is normal~\cite[\S12(v)]{Reiter71} and more generally when $H$ has Derighetti's condition as defined above~\cite[Prop.~1]{Derighetti78}. Both conditions are rather restrictive and fail e.g.\ for $H=\mathrm{SL}_2(\RR)$ in $G=\mathrm{SL}_2(\CC)$, see~\cite{Derighetti78}. Based on the criteria stated in Theorem~\ref{thm:classic}, we deduce an affirmative answer to Reiter's question:

\begin{thm}\label{thm:Reiter}
Let $G$ be a locally compact group and $H<G$ any closed subgroup. The following are equivalent:
\begin{enumerate}[label=(\roman*)]
\item $H$ is amenable.\label{pt:pReiter:amen}
\item The algebra $J^1(G,H)$ has a bounded right approximate identity.\label{pt:pReiter:alg}
\end{enumerate}
\end{thm}

Although most functional analytic methods used here are only available for topological groups when they are locally compact, the general definitions of amenability and relative amenability make sense for arbitrary topological groups. In that generality, amenability is less well-behaved; for instance, it is well-known that there are amenable Polish groups with non-amenable closed subgroups. This shows in particular that relative amenability is very different from amenability in that setting.

Here is a classical example. Let $H$ be any countable non-amenable group (with the discrete topology). Let $G$ be the unitary group of the Hilbert space $\ell^2(H)$, endowed with the strong operator topology (which coincides with the weak operator topology on $G$). Then $G$ is a Polish group and $H$ is a discrete subgroup of $G$ via the regular representation. However, Pierre de la Harpe showed that $G$ is amenable (Proposition~1(iii) in~\cite{Harpe73}, with a different terminology).

\bigskip

To wrap up this introduction, we propose a generalisation of our fixed point property to measurable actions. Let $G$ be a locally compact group with a measurable action on a standard probability space $(X, \mu)$ preserving the class of $\mu$ (i.e.\ preserving null-sets, but in general not $\mu$). We shall say that this action is \emph{relatively amenable} if for any non-empty convex compact $G$-space $K$ there is a measurable map $\Phi\colon X\to K$ which is $G$-equivariant in the sense that for all $g\in G$ we have $\Phi(g x) = g \Phi(x)$ for a.e.~$x\in X$.

This property follows if the action is amenable in Zimmer's sense~\cite[\S4.3]{Zimmer84}, assuming $G$ second countable. However, the property is a priori weaker. If $X=G/H$ (endowed with the unique invariant measure class), relative amenability of $G\acts X$ is equivalent to the relative amenability of $H$, see Proposition~\ref{prop:furst}. Here are two suggestions: (1)~Prove that if $G$ belongs to the class~$\cx$, then relative amenability of $G$-actions is equivalent to Zimmer-amenability. (2)~In general, prove that relative amenability of an action is equivalent to the conjunction of the amenability of the induced equivalence relation with the relative amenability of the stabiliser of a.e.\ point in $X$.

Note that~(2) would imply~(1) by using Theorem~A of~\cite{Adams-Elliott-Giordano} (the problem in the latter is precisely not an issue in~$\cx$, see Section~\ref{sec:measure}).

\subsection*{Structure of the paper}
After a preliminary Section~\ref{sec:notation} fixing the basic definitions and notation, we proceed to the proof of Theorem~\ref{thm:equiv} in Section~\ref{sec:char}. Section~\ref{sec:classic} collects the arguments and suitable references needed to establish Theorem~\ref{thm:classic}. Then Section~\ref{sec:closure} is devoted to the stability properties of the class $\cx$; it contains the proof of Theorem~\ref{thm:closure}. Proposition~\ref{prop:derigh} and Corollary~\ref{cor:limit} and Theorem~\ref{thm:Reiter} are proved in Section~\ref{sec:cors}, while Section~\ref{sec:further} presents a few additional observations.

\subsection*{Acknowledgements}
We thank Marc Burger for pointing out Portmann's thesis~\cite{Portmann} to us, where a related relative fixed point property was introduced and studied independently (see also \S\ref{sec:relative} below). We thank Matthias Neufang for pointing out an inaccuracy in a previous draft and indicating the reference~\cite{Hively73}. We thank the anonymous referee for comments that improved the exposition.

\section{Notation}\label{sec:notation}
All Banach spaces will be over~$\RR$, but the statements and proofs hold unchanged over~$\CC$. We use ${\langle\cdot,\cdot\rangle}$ for various duality pairings. Spaces of measurable bounded function classes are denoted by $L^\infty$, while $\cb$ denotes continuous bounded functions. We endow these spaces with the sup-norm.

\bigskip
The bounded \textbf{right uniformly continuous} functions $\ru(G)$ on a topological group $G$ are the \emph{continuous vectors} in $\cb(G)$ for the left translation representation, i.e. those vectors $f \in \cb(G)$ such that the associated orbit map $G \to \cb(G)$ is continuous. In particular, the $G$-action on $\ru(G)$ is jointly continuous and, hence, the space of means on $\ru(G)$ is a convex compact $G$-space for the weak-* topology. When $G$ is locally compact, the space $\ru(G)$  coincides with the set of continuous vectors in $L^\infty(G)$, and Cohen's factorisation theorem~\cite[16.1]{Doran-Wichmann} implies that $\ru(G)$ is exactly the set of all convolutions $\fhi*f$ with $\fhi\in L^1(G)$ and $f\in L^\infty(G)$.

We warn the reader that some authors use the opposite conventions for left and right uniform continuity. With the present convention, the space $\lu(G)$ of left uniformly continuous bounded functions is the set of continuous vectors for the right translation and has no significant interest for this paper. Keeping in mind the apparent gap between Theorem~\ref{thm:equiv}\ref{pt:equiv:mean} and Theorem~\ref{thm:classic}\ref{pt:classic:mean}, we point out that the existence of a left $H$-invariant mean on $\lu(G)$ is equivalent to amenability because a suitable right convolution provides an $H$-equivariant unital map $L^\infty(G)\to \lu(G)$. For the same reason, the existence of a left $H$-invariant mean on the space of bilaterally uniformly continuous bounded functions on $G$ is equivalent to the relative amenability of $H<G$.

\bigskip
We now recall the setting considered in 1968 by Reiter~\cite{Reiter68CRAS,Reiter68}; the facts below are presented in detail in~\cite[\S8]{Reiter-Stegeman}.

Let $G$ be a locally compact group and $H<G$ a closed subgroup. Choose left Haar measures on $G$ and on $H$; this determines a $G$-quasi-invariant measure on $G/H$. Moreover, integration over $H$ provides a $G$-equivariant continuous linear surjection
$$T\colon L^1(G) \lra L^1(G/H).$$
The kernel of $T$ is denoted by $J^1(G,H)$ and does not depend on the choices made. Notice that $J^1(G,H)$ is a closed left ideal in $L^1(G)$ and in particular itself a Banach algebra. It is moreover a right module over the algebra $L^1_0(H)$, which is by definition the kernel of the integration morphism $L^1(H)\to\RR$, but not over $L^1_0(G)$.

\bigskip
A \textbf{right approximate identity} in a normed algebra $A$ is a net $\{u_i\}_{i\in I}$ in $A$ such that $a u_i$ converges to $a$ in norm for all $a\in A$. It is said bounded if there is a bound on the norm of all $u_i$. More generally, if $M$ is a normed right $A$-module and $B\se A$ any subset, we say that $M$ has a [bounded] right approximate identity in $B$ if there is a [bounded] net $\{u_i\}_{i\in I}$ in $B$ such that $m u_i$ converges to $m$ in norm for all $m\in M$. 

It is plain how to replace right by left in these definitions, and a [bounded] \textbf{approximate identity} in a normed algebra refers to the case where both left and right conditions are satisfied. For instance, normalised densities supported on arbitrarily small neighbourhoods of the identity provide a bounded approximate identity for $L^1(G)$.

Right before the proof of Theorem~\ref{thm:Reiter}, we shall clarify the difference with the terminology used at the time of Reiter's work in the 1960s--1970s.

\bigskip
Finally, we recall that the inversion map $g\mapsto g\inv$ induces an isometric involutive anti-automorphism $f\mapsto f^*$ of the Banach algebra $L^1(G)$ given by $f^*(g) = \Delta(g\inv) f(g\inv)$ wherein $\Delta$ is the modular function.

\section{Characterising relative amenability}\label{sec:char}
\begin{flushright}
\begin{minipage}[t]{0.55\linewidth}\itshape\small
\ldots c'est ce que nous appelons un th\'eor\`eme relatif.
\begin{flushright}
\upshape\tiny
H. de Balzac, \emph{\'Etudes Analytiques},\\
in: \OE uvres compl\`etes t.~18, p.~646\\
Houssiaux, Paris (1855).
\end{flushright}
\end{minipage}
\end{flushright}
\vspace{5mm}

We proceed to the proof of Theorem~\ref{thm:equiv}. In order to clarify the logical structure of the proof, we denote by~{\ref{pt:equiv:op}$_{L^1}$} and~{\ref{pt:equiv:opno}$_{L^1}$} the statements corresponding to~\ref{pt:equiv:op} and~\ref{pt:equiv:opno} with $G$-equivariance replaced by $L^1(G)$-equivariance. The proof consists of establishing the following implications, the dashed arrows being trivial.
$$
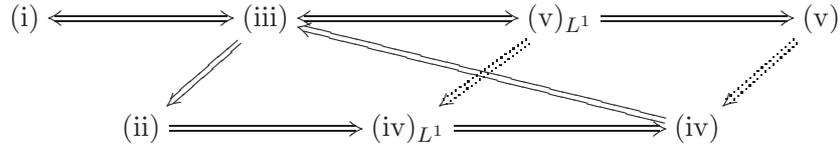
\xymatrix{
\text{\ref{pt:equiv:fixed}} \ar@{<=>}[rr] && \text{\ref{pt:equiv:mean}} \ar@{<=>}[rr]\ar@{=>}[dl] && \text{\ref{pt:equiv:opno}$_{L^1}$}\ar@{=>}[rr]  \ar@{:>}[dl] && \text{\ref{pt:equiv:opno}} \ar@{:>}[dl] \\
&\text{\ref{pt:equiv:approx}} \ar@{=>}[rr] && \text{\ref{pt:equiv:op}$_{L^1}$}\ar@{=>}[rr] && \text{\ref{pt:equiv:op}} \ar@{=>}[ulll]&
}$$
There is a pleasant surprise regarding the equivalence of conditions~\ref{pt:equiv:fixed} and~\ref{pt:equiv:mean}: the classical proof in the case $H=G$ can be used verbatim. Therefore we refer either to Rickert's original proof~\cite[Thm.~4.2]{Rickert67} or to~\cite[Thm.~3.3.1]{Greenleafbook}.

\bigskip

We denote by $\sL(L^\infty(G))$ the space of continuous linear operators of $L^\infty(G)$ and endow it with the $G$-representation by post-composition with the right translation:
$$\big( (g.\alpha)(f)\big)(x) = \alpha(f)(xg)$$
for $x,g\in G$, $f\in L^\infty(G)$, $\alpha\in \sL(L^\infty(G))$. We denote by $\sL_G(L^\infty(G))$ the invariant closed subspace of those operators that are equivariant for the left translation $G$-action, and likewise $\sL_{L^1(G)}(L^\infty(G))$ for the $L^1(G)$-action by left convolution. 

It is known (by an approximate identity argument) that $\sL_{L^1(G)} \se \sL_G$, hence the implications {\ref{pt:equiv:op}$_{L^1}$}$\Longrightarrow$\ref{pt:equiv:op} and {\ref{pt:equiv:opno}$_{L^1}$}$\Longrightarrow${\ref{pt:equiv:opno}} follow. This inclusion can however be strict, a fact going back to~\cite[\S4]{Raimi59} for $G=\RR$ and generalized in~\cite{Granirer73,Liu-vanRooij,Rudin72} (notwithstanding the incorrect~\cite{Renaud72}).

\bigskip

Consider the dual $\ru(G)^*$ endowed with the dual of the left translation action. There is a completely canonical identification (compare~\cite[p.~177]{Curtis-Figa-Talamanca})
$$ \sL_{L^1(G)}(L^\infty(G))\ \cong\ \ru(G)^*$$
wherein $\alpha\in \sL_{L^1(G)}(L^\infty(G))$ and $m\in \ru(G)^*$ correspond to each other as follows. For $f\in\ru(G)$, we set $m(f) = \alpha(f)(e)$, which makes sense since $\alpha$ must preserve continuous vectors by $G$-equivariance, that is, it preserves $\ru(G)$. In the reverse direction, $\alpha$ is defined from $m$ by
$$\langle\alpha(f), \fhi \rangle = m(f \diamond \fhi) \kern20mm(f\in L^\infty(G), \fhi\in L^1(G))$$
where $(f \diamond \fhi)(s) = \langle \ro(s)f, \fhi\rangle = \int_G \fhi(g) f(gs)\,d g$ for $s\in G$. Here $f\diamond \fhi$ is bounded (by $\|\fhi\|_1 \|f\|_\infty$) and is indeed in $\ru(G)$ (see e.g.~\cite[3.6.3]{Reiter-Stegeman}); alternatively, this is apparent from checking $f \diamond \fhi = \fhi^* * f$.

We point out that the $L^1(G)$-equivariance of $\alpha$ is used in the verification that the assignments $\alpha \leftrightarrow m$ are mutually inverse. The following properties~\ref{pt:delta}-\ref{pt:identification}-\ref{pt:normalisation} are straightforward verifications. As for~\ref{pt:GH}, it can be checked using Cohen's factorisation theorem; we shall not use it, but list it for comparison with Theorem~\ref{thm:classic}\ref{pt:classic:op}.

\begin{enumerate}[label=(\alph*)]
\item $\alpha=\id$ if and only if $m=\delta_e$.\label{pt:delta}
\item The correspondence $\alpha \leftrightarrow m$ is linear, positive, isometric and $G$-equivariant.\label{pt:identification}
\item $\alpha(\one_G)=\one_G$ if and only if $m(\one_G)=1$.\label{pt:normalisation}
\item $\alpha$ is the identity on $L^\infty(G/H)$ if and only if the restriction of $m$ to $\ru(G/H)$ is the Dirac mass at the trivial coset $H$ in $G/H$.\label{pt:GH}
\end{enumerate}

If follows from the definition of the action that $\alpha$ is $H$-fixed if and only if it ranges in $L^\infty(G/H)$. Therefore, points~\ref{pt:identification} and~\ref{pt:normalisation} above are already sufficient to conclude that the conditions~\ref{pt:equiv:mean} and~{\ref{pt:equiv:opno}$_{L^1}$} of Theorem~\ref{thm:equiv} are equivalent.

\begin{rem}\label{rem:Grothendieck}
We have tacitly used the canonical isometric identification of $L^\infty(G/H)$ with a subspace of $L^\infty(G)$. Likewise, this induces a canonical isometric identification of the space $\sL(L^\infty(G), L^\infty(G/H))$ with a subspace of $\sL(L^\infty(G))$. This inclusion fits into an exact sequence
$$0 \lra \sL\big(L^\infty(G), L^\infty(G/H)\big) \lra \sL\big(L^\infty(G)\big) \lra \sL\big(L^\infty(G), J^1(G,H)^*\big) \lra 0$$
where the epimorphism is induced via duality by the inclusion $J^1(G,H)\to L^1(G)$. This follows from the general properties of projective tensor products and $L^1$-spaces established by Grothendieck~\cite{Grothendieck55b} together with the canonical isometric identification of $\sL(L^\infty(G))$ with the dual of the projective tensor product $L^\infty(G) \hat\otimes L^1(G)$.
\end{rem}

The proof of \ref{pt:equiv:op}$\Longrightarrow$\ref{pt:equiv:mean} is in two steps. First, we notice that the above map $\alpha\mapsto m$ is in fact defined on $\sL_G(L^\infty(G))$ since we only used $G$-equivariance to justify that $\alpha(f)$ is continuous and hence can be evaluated at the point~$e$. This extended map is still linear, positive, contractive, $G$-equivariant and satisfies that $\alpha(\one_G)=\one_G$ implies $m(\one_G)=1$. The only verification that fails is that it need not be a \emph{right} inverse to the map $m\mapsto \alpha$. In any case, we obtain an $H$-invariant element $m$ of the dual of $\ru(G)$ such that $m(\one_G)=1$.

Secondly, we observe that the order structure on $\ru(G)$ (and hence on its dual) is $G$-invariant; hence we can replace $m$ by its normalised absolute value $m/|m|$. The condition $m(\one_G)=1$ is preserved since $\one_G$ is the least upper bound for the unit ball in $\ru(G)$.

Notice in passing that we have obtained an equivariant projection of $\sL_G$ onto $\sL_{L^1(G)} \cong\ru(G)^*$; for a description of $\sL_G$ itself in terms of a sort of means, see~\cite[5.3]{Hively73}.

\bigskip

We shall now establish the implication~\ref{pt:equiv:approx}$\Longrightarrow${\ref{pt:equiv:op}$_{L^1}$}. Suppose that $J^1(G,H)$ has a bounded right approximate identity $\{u_i\}_{i\in I}$ in $L^1_0(G)$. Choose an ultrafilter on $I$ dominating the order filter and denote the corresponding ultralimits by $\ulim_i$. We endow $\sL(L^\infty(G))$ with the weak-* topology coming from its identification with the dual of $L^\infty(G) \hat\otimes L^1(G)$. In particular its closed balls are compact by the Banach--Alao\u{g}lu theorem. Therefore, we can define $\alpha=\ulim_i \alpha_i$ where $\alpha_i$ is the operator determined by
$$\langle \alpha_i(f), \fhi\rangle = \langle f, \fhi - \fhi*u_i\rangle \kern20mm(f\in L^\infty(G), \fhi\in L^1(G)).$$
Then $\alpha$ is $L^1(G)$-equivariant and moreover $\alpha(\one_G) = \one_G$ since $\fhi*u_i$ is in $L^1_0(G)$. It remains only to justify that $\alpha(f)$ is right $H$-invariant. Given Remark~\ref{rem:Grothendieck}, it suffices to show that $\langle \alpha_i(f), \fhi\rangle$ converges to zero when $\fhi\in J^1(G,H)$. This follows because  $\fhi - \fhi*u_i$ goes to zero since $\{u_i\}$ is an approximate identity.

\bigskip

Finally we establish~\ref{pt:equiv:mean}$\Longrightarrow$\ref{pt:equiv:approx}. Assume that we have an $H$-invariant mean $m$ on $\ru(G)$ and consider the element $m-\delta_e$ in $\ru(G)^*$. The inclusion $\ru(G)\to L^\infty(G)$ induces a quotient map of $L^\infty(G)^*$ onto $\ru(G)^*$; therefore, Goldstine's theorem provides us with a bounded net $\{v_j\}_{j\in J}$ in $L^1(G)$ such that $\langle f, v_j\rangle$ converges to $m(f)-f(e)$ for all $f\in \ru(G)$. The particular case of $f=\one_G$ shows that $\int v_j$ converges to zero, and therefore we can assume that each $v_j$ lies in $L^1_0(G)$ upon subtracting a net converging to zero in $L^1(G)$.

We claim that for any $x\in J^1(G,H)$, the net $x*v_j - x$ converges weakly to zero in $L^1(G)$; we emphasize that weak convergence requires us to pair this net against any $q$ in $L^\infty(G)$, not just in $\ru(G)$. Let us first explain the (standard) way to conclude the proof from this claim. The ``Mazur trick'' states that from any net converging weakly to zero in a locally convex space one can construct a net of convex combinations converging to zero. Since taking convex combinations preserves boundedness and membership to $L^1_0$, we can apply this to the net $\{(x*v_j - x)_{x\in F}\}_{F, j}$ indexed by finite sets of $F$ in $J^1(G,H)$ and $j\in J$. In particular, we obtain a net of convex combinations of the form $x*u_i - x$ for some net $\{u_i\}_{i\in I}$ of convex combinations of $v_j$, finishing the proof. (The Mazur trick is a direct application of Hahn--Banach and a similar use of it can be found e.g.\ in~\cite[2.4.2]{Greenleafbook}, to which we refer for further details).

In order to prove the claim, fix $x\in J^1(G,H)$ and $q\in L^\infty(G)$. Fubini's theorem and the left invariance of the Haar measure imply
$$\langle q, x*v_j\rangle = \langle q\diamond x, v_j\rangle.$$
We noted earlier that $q\diamond x$ is in $\ru(G)$ and thus the right hand side converges to
$$(m-\delta_e)(q\diamond x) = \langle \alpha(q) - q, x\rangle$$
Since $\alpha$ is $H$-invariant, $\langle \alpha(q), x\rangle$ vanishes and hence the above expression is $\langle q, -x\rangle$, proving the claim.

\section{Proofs/references for Theorem~\ref{thm:classic}}\label{sec:classic}
Just as in Section~\ref{sec:char}, we shall denote by~{\ref{pt:classic:op}$_{L^1}$} and~{\ref{pt:classic:opno}$_{L^1}$} the statements corresponding to points~\ref{pt:classic:op} and~\ref{pt:classic:opno} of Theorem~\ref{thm:classic} with $G$-equivariance replaced by $L^1(G)$-equivariance. We begin by justifying that these four conditions are equivalent.

Indeed, although~\ref{pt:equiv:opno} is formally stronger than~\ref{pt:classic:op}, it is standard that they are equivalent: for instance the same argument as in~\cite[Lemma~5.3.7]{Monod_LN} can be applied. This holds unchanged for {\ref{pt:classic:opno}$_{L^1}$}$\Longleftrightarrow${\ref{pt:classic:op}$_{L^1}$}. Moreover, the existence of a $G$-equivariant map in this setting is equivalent to the existence of a $L^1(G)$-equivariant map for the reasons exposed in Section~\ref{sec:char}; in the present case this is stated explicitly in~\cite[Thm.~1]{Bekka90b}.

\bigskip

The equivalence of~\ref{pt:classic:approx} with amenability is due to Derighetti: Theorem~2 in~\cite{Derighetti78}.

\bigskip
Regarding~\ref{pt:classic:mean}, the equivalence of amenability with the existence of an $H$-invariant mean on $L^\infty(G)$ can be deduced e.g.\ from~\cite[\S3, Prop.~1]{Reiter71} via the usual correspondence between means on $L^\infty(G)$ and nets of positive normalised densities in $L^1(G)$. A mean on $L^\infty(G)$ restricts of course to a mean on $\cb(G)$; we should still recall why the existence of an $H$-invariant mean on $\cb(G)$ implies that $H$ is amenable: the map $\cb(H) \to \cb(G)$ as defined in the proof of Lemma~\ref{lem:open} endows $\cb(H)$ with an invariant mean.

\bigskip
The fixed point property of amenable groups immediately implies~\ref{pt:classic:fixed}. Thus, in summary, it suffices to prove~\ref{pt:classic:fixed}$\Longrightarrow$\ref{pt:classic:opno} and to prove that~\ref{pt:classic:opno} implies that $H$ is amenable. The latter is the most difficult implication of the theorem.

\bigskip
For~\ref{pt:classic:fixed}$\Longrightarrow$\ref{pt:classic:opno}, we endow $\sL(L^\infty(G))$ with the same weak-* topology and $G$-representation by right post-composition as in the above proof of Theorem~\ref{thm:equiv}. We can then consider the convex compact $G$-space $K$ of norm one positive left-$G$-equivariant maps $\alpha\in\sL(L^\infty(G))$ with $\alpha(\one_G)=\one_G$. By assumption, there is an $H$-fixed point $\alpha$ in the closed convex hull of the identity. Like the identity, the map $\alpha$ is $L^\infty(G/H)$-linear since the latter property is $H$-invariant.

\bigskip
Finally, assume that we have a map $\alpha$ as in~\ref{pt:classic:opno}. We shall prove that $H$ is amenable following the ideas of Zimmer~\cite{Zimmer77_vN, Zimmer78b}. We give a complete proof since our setting is slightly different and Zimmer dealt only with discrete groups in~\cite{Zimmer77_vN}, but all the main ideas are taken from Zimmer's work. In order to use ergodic theory, we need to assume for now that $G$ is second countable and we shall indicate at the end of the proof how to reduce to that case.

Our goal is to find an $H$-invariant mean on $\ru(H)$. Since $H$ is second countable, $\ru(H)$ is the directed union of its \emph{separable} closed $H$-invariant subspaces and by a compactness argument it suffices to find an $H$-invariant mean on any such separable subspace $E$. Let $E^*$ be the contragredient module; we endow the set $K\se E^*$ of means with it weak-* topology, turning it into a (non-empty) convex compact $H$-space. We recall that the space $\lw(G, E^*)$ of weak-* measurable bounded function classes on $G$ is dual to the space $L^1(G, E)$ of Bochner-integrable function classes (and likewise on $G/H$). We can define an operator
$$\alpha_E\colon \lw(G, E^*)\lra  \lw(G/H, E^*)$$
by $\langle\alpha_E f, v\rangle= \alpha(\langle f, v\rangle)$ for $v\in E$ since there is a canonical identification of $\lw(G/H, E^*)$ with $\sL(E, L^\infty(G/H))$.

Let $\sigma\colon G/H\to G$ be a Borel section of the projection $G\to G/H$ and define a Borel cocycle $\gamma\colon G\times G/H\to H$ by $\gamma(g,x)=\sigma(gx)\inv g\sigma(x)$. Choose a point $k_0\in K$ and define the map $\psi\colon G\to K$ by $\psi(g) = \sigma(gH)\inv g k_0$, noting that $\sigma(gH)\inv g$ lies in $H$. Then $\psi(\bar g g) = \gamma(\bar g, gH) \psi(g)$ for all $\bar g, g$. We thus obtain an element $\alpha_E\psi\in \lw(G/H, E^*)$. The fact that $\alpha$ is positive implies that $\alpha_E\psi$ still ranges a.e.\ in $K$; indeed it suffices to compose $\psi$ (respectively $\alpha_E\psi$) with the evaluation on a countable dense set of elements $v\in E$ that separate $K$ from any other point in $E^*$.

We claim that $\alpha_E\psi$ is $\gamma$-equivariant in the sense that for any $\bar g\in G$, the equality $\alpha_E\psi(\bar g g H) = \gamma(\bar g, gH) \alpha_E\psi(g H)$ holds for a.e.~$gH$. To this end, we first show that any bounded Bochner-measurable map $V\colon G/H\to E$ satisfies $\langle\alpha_E \psi, V\rangle= \alpha(\langle \psi, V\rangle)$ a.e.\ on $G/H$. Since $V$ is Bochner-measurable, it suffices to verify it for all functions of the form $V=v \one_{A}$ for $v\in E$ and a measurable set $A\se G/H$ considered also as $H$-invariant subset of $G$. Then the definition of $\alpha_E$ together with the $L^\infty(G/H)$-linearity of $\alpha$ implies
$$\langle\alpha_E \psi, v \one_{A}\rangle = \langle \one_A\alpha_E \psi, v\rangle = \langle \alpha_E (\one_A \psi), v \rangle =  \alpha(\langle \one_A \psi, v\rangle) =  \alpha(\langle \psi, v \one_A \rangle)$$
as was to be shown. Now we prove the claim; fix some $\bar g\in G$. We need to show that for all $v\in E$ we have $\langle \lambda(\bar g)(\alpha_E f), v\rangle= \langle \gamma(\bar g, -H) \alpha_E f, v\rangle$ a.e.\ on $G/H$, where $\lambda$ is the left translation representation. By equivariance of $\alpha$, the left hand side is $\alpha(\langle \lambda(\bar g)\psi, v\rangle)$, which is $\alpha(\langle \gamma(\bar g, -H)  \psi, v\rangle)$, that is $\alpha(\langle \psi,  \gamma(\bar g, -H) \inv v\rangle)$. In other words, we need to show $\alpha(\langle  \psi, \gamma(\bar g, -H)\inv v\rangle) =  \langle\alpha_E \psi ,  \gamma(\bar g, -H)\inv v\rangle$, which holds indeed by taking 
$$V(gH):= \gamma(\bar g, gH)\inv v$$ 
in the statement above.

We finally define a weak-* measurable map $\fhi\colon G\to K$ by $\fhi(g) = g \sigma(g\inv H) \alpha_E\psi(g\inv H)$; here $g \sigma(g\inv H) \in H$. The claim implies that for all $\bar g\in G$ we have $\fhi(g\bar g) = \fhi(g)$ for a.e.~$g\in G$. Choosing a countable dense subgroup $\Lambda<G$, we have a conull set of $g\in G$ for which $\fhi(g\bar g) = \fhi(g)$ holds for all $\bar g\in \Lambda$. Since $E$ is separable, it follows from the ergodicity of the $\Lambda$-action on $G$ that there is $k\in K$ with $\fhi(g)=k$ for a.e.~$g\in G$. But by construction, we have $\fhi(hg) = h\fhi(g)$ for all $h\in H$ and a.e.~$g\in G$; therefore, $k$ is an $H$-fixed point in $K$.

\bigskip
It remains only to justify how to reduce to the case where $G$ is second countable. We first claim that we can assume $G$ compactly generated. Indeed, let $G_U<G$ be the subgroup generated by some compact neighbourhood $U\se G$ of the identity and set $H_U = H\cap G_U$. Then the compactly generated subgroup $G_U$ is open in $G$. Therefore, $L^\infty(G_U)$ can be viewed as a $G_U$-equivariantly complemented subspace of $L^\infty(G)$, and likewise for $L^\infty(G_U/H_U)$ in $L^\infty(G/H)$. To prove the claim, we assume that there is a $G$-equivariant conditional expectation from $L^\infty(G)$ to $L^\infty(G/H)$. The $L^\infty(G/H)$-linearity implies that the image of $L^\infty(G_U)$ is in $L^\infty(G_U/H_U)$, as seen by multiplying with $\one_{G_U H}$. Thus we have in particular a $G_U$-equivariant conditional expectation from $L^\infty(G_U)$ to $L^\infty(G_U/H_U)$. Supposing the compactly generated case settled, we conclude that $H_U$ is amenable. As $U$ varies, the groups $H_U$ form a directed family whose union is $H$, which implies that $H$ is itself amenable as claimed.

\smallskip
We thus assume $G$ compactly generated and reduce to the second countable case. Since $G$ is in particular $\sigma$-compact, it admits a compact normal subgroup $K\triangleleft G$ with $G/K$ second countable, see~\cite[Satz~6]{KakutaniKodaira}. It suffices to prove that $HK/K$ is amenable. However, any $G$-equivariant conditional expectation from $L^\infty(G)$ to $L^\infty(G/H)$ restricts to a $G/K$-equivariant conditional expectation from $L^\infty(G/K)$ to $L^\infty(G/HK)$ since $K$ is normal. This finishes the reduction to the second countable case.

\section{Stability properties of the class $\cx$}
\label{sec:closure}
We first record the elementary \emph{proof of Proposition~\ref{prop:clarify}}. By design, if $G$ admits a non-empty convex compact $G$-space with amenable stabilisers, then any relatively amenable subgroup is amenable. Suppose conversely that relative amenability implies amenability for all closed subgroups of $G$. Then Condition~\ref{pt:equiv:mean} of Theorem~\ref{thm:equiv} shows that the space of means on $\ru(G)$ has only amenable stabilisers.\qed

\medskip
The proof of Theorem~\ref{thm:closure} requires a number of preparations. We start with the following property of amenable actions.

\begin{prop}\label{prop:moy-convexe}
Let $G$ be a locally compact group with a continuous action on a compact space $Z$. Let $K\se \cc(Z)^*$ be the convex compact $G$-space of probability measures on $Z$ with the weak-* topology.

Then the $G$-action on $Z$ is amenable if and only if the $G$-action on $K$ is amenable.
\end{prop}

\begin{proof}
We recall that the natural $G$-map $Z\to K$ defined by Dirac masses is a homeomorphism on its image; in particular, the action on $Z$ is amenable if the action on $K$ is so. For the converse, suppose that the action on $Z$ is amenable. We shall work with the criterion of~(3) in Proposition~2.2 in~\cite{Anantharaman02}. That is, there is a net $\{g_i\}_{i\in I}$ of compactly supported functions $g_i\colon Z\times G\to\RR_+$ such that

\smallskip
\begin{enumerate}[label=(\arabic*), itemsep=1ex]
\item $\lim_i \sup_{z\in Z} \Big|1-\int_G g_i(z,t)\,dt \Big| =0$,\label{pt:moy-c1}
\item $\lim_i \sup_{z\in Z, s\in S} \int_G \big| g_i(sz,st) - g_i(z,t)\big| \,dt =0$ for any compact $S\se G$.\label{pt:moy-c2}
\end{enumerate}

\smallskip

We can extend $g_i$ by convexity to $\widetilde g_i\colon K\times G\to\RR_+$ by setting $\widetilde g_i(\mu, s) = \int_Z g_i(z,s)\, d\mu(z)$ for $\mu\in K, s\in G$. The integral makes sense since $g_i$ is continuous; moreover $\widetilde g_i$ remains continuous  and compactly supported. The two conditions~\ref{pt:moy-c1} and~\ref{pt:moy-c2} above are inherited by $\widetilde g_i$ because of the built-in uniformity of the convergence.
\end{proof}

The following consequence will be relevant. 

\begin{cor}\label{cor:moduloamen}
Let $G$ be a locally compact group amenable at infinity and $N\lhd G$ an amenable closed normal subgroup. Then $G/N$ is amenable at infinity.
\end{cor}

\begin{proof}
By Proposition~\ref{prop:moy-convexe}, there is a non-empty convex compact $G$-space $K$ on which $G$ acts amenably. Since $N$ is amenable, the convex compact $G$-space of $N$-fixed points $K^N$ is non-empty; it is still amenable since it is a subspace of $K$. On the other hand, it is a $G/N$-space; its amenability as $G$-space implies its amenability as $G/N$-space. (The latter fact is most apparent if one uses the Definition~2.1 in~\cite{Anantharaman02} and projects measures from $G$ to $G/N$.)
\end{proof}

We next record that relative amenability is inherited by restriction to open subgroups. 

\begin{lem}\label{lem:open}
Let $G$ be a locally compact group, $H<G$ a closed subgroup and $O<G$ an open subgroup. If $H$ is amenable relative to $G$, then $H\cap O$ is amenable relative to $O$.
\end{lem}

\begin{proof}
At first we shall only use that $O$ is a closed subgroup. Choose a left Haar measure on $O$ and let $\beta\colon G\to \RR_+$ be a Bruhat function. Recall that this means a continuous function with the following two properties: (1)~for each compact subset $Q\se G$, the support of $\beta$ meets $OQ$ in a compact subset; (2)~for all $g\in G$ one has $\int_O \beta(o\inv g)\,do = 1$. Bruhat functions exist, see e.g.~\cite[Ch.~8 \S1.9]{Reiter68} (with a different notation). Given $f\in \cb(O)$, we define a function $\tilde f$ on $G$ by $\tilde f(g) = \int_O \beta(o\inv g) f(o)\, do$. Then the map $f\mapsto \tilde f$ is a positive norm one $O$-equivariant linear map $\cb(O) \to \cb(G)$ with $\widetilde{\one_O} = \one_G$. (For the verification that  $\tilde f$ is indeed continuous, see e.g. the proof of Proposition~1.12~\cite{Paterson}.)  

We now show that $\tilde f\in \ru(G)$ whenever $f\in\ru(O)$; this is where we shall use that $O$ is open in $G$. We need to prove the continuity of the orbit map $G\to L^\infty(G)$ associated to $\tilde f$. Since $O$ is open, we can restrict this orbit map to $O\to L^\infty(G)$. By hypothesis the orbit map $O \to L^\infty(O)$ associated to $f$ is continuous; the desired assertion follows since the map $f \mapsto \tilde f$ is $O$-equivariant and of norm one.

Finally, the last assertion of the Lemma follows from Condition~\ref{pt:equiv:mean} of Theorem~\ref{thm:equiv}.
\end{proof}

It is well-known that amenability is stable under forming group extensions. The following lemma is a slight variation of that fact. 

\begin{lem}\label{lem:co-amen}
Let $G, H$ be   locally compact groups, $O < G$ be an open subgroup and  $\varphi \colon H \to G$ be a continuous homomorphism.  Then $H$ is amenable if and only if $\overline{\varphi(H)}$ and $\varphi\inv(O)$ are both amenable. 
\end{lem}

\begin{proof}
The `only if' part is clear.
Assume conversely that $P=\varphi\inv(O)$ and $\overline{\varphi(H)}$ are both amenable. In order to deduce that $H$ is amenable, it suffices to show that $P$ is \emph{co-amenable} in $H$, i.e. that the transitive $H$-action on $H/P$ preserves a mean  (see e.g.~\cite{Eymard72,Monod-Popa} for more on this notion). By hypothesis $G_1  =\overline{\varphi(H)}$ is amenable, hence the $G_1$-action on $G_1/O_1$ preserves a mean, where $O_1 = G_1 \cap O$. The desired assertion follows since the homomorphism $\varphi \colon H \to G_1$ has dense image, and hence induces an $H$-equivariant bijection of discrete sets $H/P \to G_1/O_1$. 
\end{proof}

We record the following subsidiary fact, which relies on a combination of the previous two lemmas. 

\begin{lem}\label{lem:UN}
Let $G$ be a locally compact group, $N \lhd G$ be a closed normal subgroup, and $O< G$ be an open subgroup containing $N$. If $G/N$ and $O$ both belong to $\cx$, then so does $G$. 
\end{lem}

\begin{proof}
Let $\varphi \colon G \to G/N$ denote the canonical projection. Let $H < G$ be a closed, relatively amenable subgroup. Then $\overline{\varphi(H)}$ is relatively amenable in $G/N$, and thus amenable since $G/N \in \cx$. Moreover, the intersection $H \cap O$ is relatively amenable in $O$ by Lemma~\ref{lem:open}, and hence amenable since $O \in \cx$. It follows from Lemma~\ref{lem:co-amen} that $H$ is amenable, as desired. 
\end{proof}

\begin{proof}[Proof of Theorem~\ref{thm:closure}\ref{pt:class:discrete}--\ref{pt:class:open}]
Relative amenability is equivalent to amenability when $G$ is discrete: this is apparent e.g.\ by comparing Condition~\ref{pt:equiv:mean} in Theorems~\ref{thm:equiv} and~\ref{thm:classic} respectively.  This proves~\ref{pt:class:discrete}. We now assume that $G$ is amenable at infinity and proceed to show that the relative amenability of $H<G$ implies that $H$ is amenable. By assumption there exists a compact space $Z$ with a continuous amenable $G$-action. Relative amenability implies that $Z$ carries an $H$-invariant probability (Radon) measure. It follows that $H$ is amenable, since the full stabiliser in $G$ of every probability measure on $Z$ is amenable by Example~2.7(2) in~\cite{Anantharaman02} (the latter fact can alternatively be deduced from Proposition~\ref{prop:moy-convexe}). 

\medskip
The case~\ref{pt:class:sub} of subgroups follows from the definitions. For~\ref{pt:class:direct}, consider a relatively amenable closed subgroup $H$ in a product $G_1\times G_2$. Let $H_i$ be the closure of the projection of $H$ to $G_i$. Then relative amenability still holds for the subgroup $H_i$ of $G_i$; thus both $H_i$ are amenable and hence so is $H_1\times H_2$. Since $H$ is a closed subgroup of the latter, $H$ is amenable.

\smallskip
The case~\ref{pt:class:adele} of adelic products follows from the combination of~\ref{pt:class:direct} with~\ref{pt:class:union}, to which we now turn. Assume that $G$ is the union of a directed family $\{ G_i\}_{i\in I}$ of open subgroups and let $H<G$ be a relatively amenable closed subgroup. Then $H$ is the union of the groups $H\cap G_i$, each of which being relatively amenable in $G_i$ by Lemma~\ref{lem:open}. We are thus assuming that each $H\cap G_i$ is amenable, which implies that $H$ is amenable since the family $\{H\cap G_i\}_{i \in I}$ is directed.

\medskip
From now, we consider a locally compact group $G$ with a closed normal subgroup $N\lhd G$.

\smallskip
For~\ref{pt:class:ameneq} we assume $N$ amenable. Any convex compact $G/N$-space with amenable stabilisers still has amenable stabilisers as a $G$-space. Conversely, assume that there is a non-empty convex compact $G$-space with amenable stabilisers. Then the convex compact $G/N$-space of its $N$-fixed points is non-empty since $N$ is amenable and it follows $G/N\in\cx$.

\smallskip
We prove~\ref{pt:class:open} before~\ref{pt:class:conneq}; assume $N$ open. In particular $G/N$ is discrete, hence contained in $\cx$ by \ref{pt:class:discrete}. If $N \in \cx$, then $G \in \cx$ by Lemma~\ref{lem:UN} (applied  in the special case $N = O$).
The converse implication is a special case of~\ref{pt:class:sub}.

\smallskip
For~\ref{pt:class:conneq} we assume $N$ connected. Suppose first $G/N\in\cx$. Let $R=\Ramen(G)\lhd G$ be the amenable radical of $G$ and recall that $G$ has a finite index open characteristic subgroup $G^*\lhd G$ containing $R$ such that $G^*/R$ is a direct product $G^*/R \cong D\times S$ with $D$ totally disconnected and $S$ a connected semi-simple Lie group (Theorem~11.3.4 in~\cite{Monod_LN}). By point~\ref{pt:class:ameneq}, the quotient $G/\overline{NR}$ is in~$\cx$. The image of $G^*$ in the latter group is a finite index open subgroup of the form $D\times S_1$ for some quotient $S_1$ of $S$; thus $D$ appears as a closed subgroup in $G/\overline{NR}$ and hence $D\in\cx$ by~\ref{pt:class:sub}. Since $S$ is amenable at infinity~\cite[3.2(3)]{Anantharaman02}, we conclude by~\ref{pt:class:direct} that $D\times S$ is in~$\cx$. Now $G^*\in\cx$ by~\ref{pt:class:ameneq} and finally $G\in\cx$ by~\ref{pt:class:open}.

Conversely, suppose $G\in\cx$ and keep the above notation. By~\ref{pt:class:ameneq}, the quotient $G/R$ is in~$\cx$ and hence so is $D$ by~\ref{pt:class:sub}. Since $(G/R)/S$ contains $D$ as an open normal subgroup of finite index, it is in~$\cx$ by~\ref{pt:class:open}. Since $G/\overline{NR}$ is an extension of $(G/R)/S$ by a connected kernel, it is in~$\cx$ by the first implication. Finally, $G/N$ is an extension of $G/\overline{NR}$ by an amenable kernel and hence we conclude by~\ref{pt:class:ameneq}.
\end{proof}

In order to establish the last two points of  Theorem~\ref{thm:closure}, a few additional tools are needed; their proofs rely on the  assertions from Theorem~\ref{thm:closure} which have already been proven.

\begin{lem}\label{lem:ext}
Let $\cx_0$ be a subclass of $\cx$ enjoying the following stability properties, where $G$ denotes a locally compact group and $N < G$ a closed normal subgroup:
\begin{enumerate}[label=(\arabic*)]
\item If $G \in \cx_0$, so does any closed subgroup of $G$. \label{it:closed}

\item If $G \in \cx_0$ and $N$ is amenable, then $G/N \in \cx_0$. \label{it:ramen}

\item If $N \in \cx_0$ and $G/N$ is compact, then $G \in \cx_0$. \label{it:cpt}
\end{enumerate}

Then $\cx$ enjoys the following stability property: if $N \in \cx_0$ and $G/N \in \cx$, then $G \in \cx$. 
\end{lem}

\begin{proof}
Let $G$ be a locally compact group and $N < G$ be a closed normal subgroup such that $N \in \cx_0$ and $G/N \in \cx$. We need to show that $G \in \cx$. 

Let us assume in a first case that $G$ is totally disconnected. Then it has a compact open subgroup $U$ by van Dantzig's theorem. The product $O = NU$ is open in $G$ and belongs to $\cx_0$ by~\ref{it:cpt}, hence to $\cx$.  By Lemma~\ref{lem:UN}, we infer that $G \in \cx$, as desired.

We now turn to the general case. 
By the solution to Hilbert's fifth problem, the connected locally compact group $(G/N)^\circ$ has a unique maximal compact normal (hence characteristic) subgroup $W$ such that  $(G/N)^\circ/W$ is a Lie group. Let $M$ be the pre-image of $W$ in $G$. Then $M/N$ is compact, so that $M \in \cx_0$ by~\ref{it:cpt}. By Theorem~\ref{thm:closure}\ref{pt:class:ameneq} we have $G/M \in \cx$ and by construction the identity component of $G/M$ is a Lie group. 

Let now $R$ denote the amenable radical of $M$. We have $M/R \in \cx_0$ by~\ref{it:ramen} and, in view of Theorem~\ref{thm:closure}\ref{pt:class:ameneq}, the desired conclusion that $G\in \cx$ will follow if one shows that $G/R \in \cx$. Hence we may assume without loss of generality that $R=1$. 

By Theorem~11.3.4 in~\cite{Monod_LN}, the group $M$ has an open characteristic subgroup of finite index $M^*$ isomorphic to a direct product $S  \times D$, where $S = M^\circ$ is a connected semi-simple Lie group with trivial centre and no compact factor, and $D = \centra_M(M^\circ)$ is a totally disconnected group with trivial amenable radical. In particular $D$ is characteristic in $M$, hence normal in $G$. Notice moreover that $D \in \cx_0$ by~\ref{it:closed}.

We next claim that $G^\circ \cap D = 1$. 
Indeed, the intersection    $G^\circ \cap D$ is a totally disconnected closed normal subgroup of $G^\circ$. Invoking again the solution to Hilbert's fifth problem, we find a compact normal subgroup $V$ of $G^\circ$ such that $G^\circ/V$ is a Lie group. Since the canonical projection $G^\circ \to G^\circ /V$ is proper, the image of $G^\circ \cap D$ in the connected Lie group $G^\circ/V$ is a closed totally disconnected normal subgroup. It must thus be discrete, hence central. It follows that $G^\circ \cap D$ is compact-by-\{discrete abelian\}. Therefore  $G^\circ \cap D$ is amenable, hence contained in the amenable radical of $D$, which is trivial. The claim stands proven. 

Since $(G/M)^\circ$ is a Lie group, the image of $G^\circ$ under the canonical projection $G \to G/M$ coincides with $(G/M)^\circ$ (see Lemma~2.4 in~\cite{CCMT_JEMS_pre}). In particular $G^\circ M$ is closed in $G$. It follows that $G^\circ D$ is also closed. Therefore the image of $D$ in the quotient  $G_1 = G/G^\circ$ is a closed normal subgroup $D_1$ isomorphic to $D$. In particular $D_1 \in \cx_0$. Moreover, we have $G_1/D_1 \cong G/ G^\circ D \cong (G/M^*)/ (G/M^*)^\circ \in \cx$ since $G/M \in \cx$, by using Theorem~\ref{thm:closure}\ref{pt:class:ameneq} and~\ref{pt:class:conneq}. Since the totally disconnected case has already been treated, we infer that $G_1 \in \cx$, and finally that $G \in \cx$ by  Theorem~\ref{thm:closure}\ref{pt:class:conneq}.
\end{proof}

\begin{lem}\label{lem:amen-by-discrete}
Let $\cx_0$ be the class of locally compact groups that are directed unions of amenable-by-discrete open subgroups. Then $\cx_0 $ is contained in $\cx$ and satisfies conditions~\ref{it:closed}, \ref{it:ramen} and~\ref{it:cpt} from Lemma~\ref{lem:ext}.
\end{lem}

\begin{proof}
Every amenable-by-discrete group belongs to $\cx$ by Theorem~\ref{thm:closure}\ref{pt:class:discrete} and~\ref{pt:class:ameneq}. Therefore $\cx_0 \subset \cx$ by Theorem~\ref{thm:closure}\ref{pt:class:union}. 

It follows from the definition that the class $\cx_0$ is closed under passing to closed subgroups, and to quotients by closed normal subgroups, so that conditions~\ref{it:closed} and~\ref{it:ramen}   from Lemma~\ref{lem:ext} are satisfied.

Let now $G$ be a locally compact group with a closed cocompact normal subgroup $N \in \cx_0$. We consider the following set of closed subgroups of $G$:
$$
\sF = \{G_1 < G \; | \; G_1 \text{ is open, compactly generated, and } G = G_1 N\}.
$$
Then $\sF$ is a directed set. Moreover, since $G/N$ is compact, the set $\sF$ is non-empty and we have  $G = \bigcup \sF$. Therefore, it suffices to show  that each $G_1 \in \sF$ belongs to $\cx_0$. 

Given $G_1 \in \sF$, we have $N_1 = G_1 \cap N \in \cx_0$. Moreover $G_1 / N_1 = G_1 N / N = G/N$ is compact. Since $G_1$ is compactly generated, so is thus $  N_1$ by~\cite{Macbeath-Swierczkowski59}. Any compactly generated group in $\cx_0$ is amenable-by-discrete. So is thus  $ N_1$. In other words, the amenable radical $R$ of $N_1$ is open in $N_1$. All we need to show is that the amenable radical of $G_2 = G_1/R$ is open. The image $N_2 = N_1/R$ of $N_1$ in $G_2$ is a finitely generated discrete cocompact normal subgroup with trivial amenable radical. Its centraliser $Z = \centra_{G_2}(N_2)$ is thus open, and the intersection $Z \cap N_2 = \centra(N_2) $ is trivial. Therefore $Z N_2 \cong Z \times N_2$ is an open normal subgroup of $G_2$. Since $G_2/N_2$ is compact, it follows that $Z$ is compact, hence amenable. The amenable radical of $G_2$ therefore contains $Z$, and is thus open as desired. 
\end{proof}

\begin{proof}[End of proof of Theorem~\ref{thm:closure}]
Assertion~\ref{pt:class:td:discrete} is now immediate from Lemmas~\ref{lem:ext} and~\ref{lem:amen-by-discrete}. 
For~\ref{pt:class:td:mai}, we note that the class $\cx_0$ of locally compact groups that are amenable at infinity is contained in $\cx$ by Theorem~\ref{thm:closure}\ref{pt:class:mai}.  Moreover $\cx_0$ satisfies the three conditions of Lemma~\ref{lem:ext} below:  \ref{it:closed}~is to be found e.g.\ in~\cite[5.2.5(i)]{Anantharaman-Renault},  \ref{it:ramen}~follows from Corollary~\ref{cor:moduloamen}, and~\ref{it:cpt} is ensured by~\cite[5.2.5(ii)]{Anantharaman-Renault}. The desired conclusion follows. 
\end{proof}

\begin{rem}
We have proved an assertion stronger than Theorem~\ref{thm:closure}\ref{pt:class:td:discrete}, namely:

\itshape
Suppose that $N$ is a directed union of subgroups that are open (in $N$) and amenable-by-discrete. If $G/N\in\cx$, then $G\in\cx$.
\upshape
\end{rem}

\section{Proof of corollaries}
\label{sec:cors}

A first application of Theorem~\ref{thm:equiv} is:

\begin{proof}[Proof of Proposition~\ref{prop:derigh}]
Let $H < G$ be relatively amenable and suppose that it satisfies Derighetti's condition; explicitly, this means that there are $H$-almost-invariant vectors in the quasi-regular $G$-representation on $L^2(G/H)$. This implies that there is an $H$-invariant mean on $L^\infty(G/H)$: the argument given e.g.\ in~\cite[p.~29]{Eymard72} for the case of $G$-invariant means applies without any change: the mean is a weak-* accumulation point of the densities constructed by squaring elements of $L^2(G/H)$. Composing such a mean with the map~$\alpha$ provided by Theorem~\ref{thm:equiv}\ref{pt:equiv:opno} provides a mean as required by Theorem~\ref{thm:classic}\ref{pt:classic:mean}.
\end{proof}

\begin{rem}
It follows from Reiter's characterisation of amenability that, if the closed subgroup $H < G$ is amenable, then $H$ satisfies Derighetti's condition. Therefore, the converse of Proposition~\ref{prop:derigh} holds and we obtain the following equivalence: \emph{Among the relatively amenable subgroups of $G$, amenability is equivalent to Derighetti's condition.}
\end{rem}

We now turn to Chabauty limits of amenable groups. Using Fell's multivariate continuity~\cite{Fell64}, Schochetman proved that a limit of a net amenable subgroups remains amenable when the net consists of \emph{subgroups of the limit}~\cite[Thm.~3]{Schochetman71}; this is however completely trivial with the fixed point definition of amenability. It follows readily that amenability passes to the limit when the limit is an \emph{open} subgroup~\cite[Cor.~1]{Schochetman71}.

As far as we know, the general case remains unknown. We can however answer the question whenever the ambient group $G$ belongs to the very large class~$\cx$, or when the limit group falls within the scope of Proposition~\ref{prop:derigh}. Indeed, the \emph{proof of Corollary~\ref{cor:limit}} follows immediately from:

\begin{lem}\label{lem:chabauty}
In any locally compact group, the set of relatively amenable closed subgroups is Chabauty-closed. 
\end{lem}

\begin{proof}
Let $G$ be a locally compact group, $H<G$ a closed subgroup and $K$ a non-empty convex compact $G$-space. Suppose that $H$ is the limit of a net $\{H_i\}_{i\in I}$ of closed subgroups that are relatively amenable. If $x_i\in K$ is fixed by $H_i$, then any accumulation point of the net $\{x_i\}_{i\in I}$ will be fixed by $H$.
\end{proof}

Before turning to Theorem~\ref{thm:Reiter}, we clarify the use of terminology. At the time of Reiter, what was called right approximate units (and confusingly sometimes right approximate identities) was the existence \emph{for each $x$} of a net $\{u_i\}$ such that $x*u_i\to x$. What is now called a right approximate identity (i.e. a net independent of $x$) was called \emph{multiple} approximate units. However, it was realized in 1971 that both concepts coincide in any Banach algebra, even preserving the control of the norm of the net~\cite{Wichmann73, Altman72, Altman73}.

In the present article, we shall never consider the older approximate units, noting that the equivalence is not clear when we look at subsets of algebras or at modules. Nonetheless, the above equivalence should be kept in mind when we quote Reiter's work which predates it.

\begin{proof}[Proof of Theorem~\ref{thm:Reiter}]
We begin by recalling a general fact for any Banach algebra $A$ admitting a bounded right approximate identity and any closed left ideal $J$ in $A$: the ideal $J$ has a bounded right approximate identity (in itself) if and only if its annihilator
$$J^\perp := \big\{f\in A^* : \langle f, u\rangle =0 \ \forall u\in J \big\}$$
is right-invariantly complemented in the dual $A^*$. The latter means that there is $\alpha\in\sL(A^*)$ with $\alpha(A^*)=J^\perp$, which is the identity on $J^\perp$, and which is equivariant with respect to the dual right $A$-module structure (which preserves $J^\perp$ since $J$ is a left ideal). This fact was established in~\cite[4.1.4 p.~42]{ForrestPHD}, see also~\cite[6.4]{Forrest90}. A more precise statement under the stronger asssumption that $A$ has a bounded approximate identity can be found in~\cite[\S2]{Delaporte-Derighetti}. We apply the general statement to $A=L^1(G)$; in this special case, it is also established in~\cite[Thm.~1]{Bekka90b}.

We observe that the usual left $A$-convolution on $A^*\cong L^\infty(G)$ corresponds to the dual right $A$-action via the canonical involution of $A$ recalled in Section~\ref{sec:notation}. It follows that right $A$-equivariance is equivalent to the usual left $L^1(G)$-equivariance for $\sL(L^\infty(G))$. We now consider $J=J^1(G,H)$. By the duality principle, the annihilator $J^\perp$ is identified with $L^\infty(G/H)$. Therefore, an application of the $L^1(G)$-equivariant version of criterion~\ref{pt:classic:op} in Theorem~\ref{thm:classic} finishes the proof.
\end{proof}

\section{Further observations}\label{sec:further}
\subsection{Measurable actions}\label{sec:measure}
At the end of the introduction, we defined relative amenability for actions. The following proposition shows that it generalises indeed the relative amenability of subgroups.

\begin{prop}\label{prop:furst}
Let $H$ be a closed subgroup of a second countable locally compact group $G$ and endow $G/H$ with the unique (non-zero) $G$-invariant measure class. Then the action of $G$ on $G/H$ is relatively amenable if and only if $H$ is amenable relative to $G$ .
\end{prop}

\begin{proof}
We suppose that the action is relatively amenable (the reverse implication being trivial). Let $K$ be a non-empty convex compact $G$-space; we need to find an $H$-fixed point. Since $G$ is second countable, we can assume that $K$ is separable. Let $\Phi\colon G/H\to K$ be a measurable equivariant map. We argue similarly to the last part of the proof of Theorem~\ref{thm:classic}: define $\fhi\colon G\to K$ by $\fhi(g) = g \Phi(g\inv H)$. Then, for all $\bar g\in G$ and all $h\in H$, we have $\fhi(g \bar g) = \fhi(g)$ and $\fhi(hg) = h\fhi(g)$ for a.e.~$g\in G$. The conclusion follows as in Theorem~\ref{thm:classic}.
\end{proof}

As a corollary, we see that any example of a non-amenable relatively amenable subgroup would also show that relative amenability of actions is strictly weaker than Zimmer-amenability.

\medskip

According to Theorem~A in~\cite{Adams-Elliott-Giordano}, the stabiliser of almost every point in a Zimmer-amenable action of a second countable locally compact group $G$ is an amenable (closed) subgroup of $G$. It seems that the proof contains a gap. Specifically, Lemma~4.3 in this reference is equivalent to stating that a subgroup of $G$ is amenable if and only if it has an invariant mean on $\ru(G)$, which is equivalent to relative amenability. The point in that proof that seems not to be justified is the reference (on page~816) to Proposition~7.2.7 in~\cite{Zimmer84}; indeed, that proposition uses $G$-invariance. However, this issue disappears if $G$ belongs to the class~$\cx$.

In any case, a general result for groupoids (Corollary~5.3.33 in~\cite{Anantharaman-Renault}) implies the statement of Theorem~A in~\cite{Adams-Elliott-Giordano}.

\subsection{On the Reiter condition} 

One of the well-known equivalent characterizations of the amenability of a locally compact group $G$ is the Reiter condition, namely the existence of \emph{asymptotically invariant} elements in $L^1(G)$. More precisely, a net $\fhi_i$ of positive normalized elements in $L^1(G)$ such that $g\fhi_i-\fhi_i$ goes to zero in norm for all $g\in G$.  By the Mazur trick, it suffices to have \emph{weak} convergence to zero.  Moreover, as recalled above, the amenability of $H<G$ is equivalent to requiring either form of convergence for all $g\in H$.

\medskip
How, then, can we reformulate relative amenability for $H<G$ in terms of a Reiter condition? It is not hard to check that this amounts to the convergence of $g\fhi_i-\fhi_i$ to zero (for all $g\in H$) with respect to the right \emph{strict topology} of $L^1(G)$. The latter is a locally convex topology going back to~\cite{Buck52} in the commutative case and extended to general Banach algebras in~\cite{Sentilles-Taylor}. The verification is a direct application of Cohen's factorisation.

\subsection{Pairs of subgroups}\label{sec:relative}
The fixed point property for subroups that we called relative amenability can be seen as a particular case of the following.

\begin{defn}\label{defn:pairs}
Let $G$ be a topological group and let $L, H<G$ be subgroups. We say that $L$ is co-amenable to $H$ relative to $G$ if any convex compact $G$-space with an $L$-fixed point has an $H$-fixed point. We write $L\succeq_G H$.
\end{defn}

Thus, $\succeq_G$ defines a pre-order on the family of subgroups of $G$ which descends to conjugacy classes of subgroups. At one end, $1\succeq_G H$ amounts to the relative amenability of $H<G$. At the other extreme, $L\succeq_G G$ amounts to the \emph{co-amenability} of $L<G$ as studied by Eymard~\cite{Eymard72} (see also~\cite{Monod-Popa} for more equivalent conditions). It is straightforward that $\succeq_G$ is closed in the second variable for the Chabauty topology.

\medskip
It was pointed out to us by Marc Burger that Definition~\ref{defn:pairs} was independently introduced and studied by J\"urg Portmann in his thesis~\cite{Portmann}. Moreover, the following Rickert-like characterisation is proved in~\cite[2.3.5]{Portmann}: $L\succeq_G H$ if and only if there is an $H$-invariant mean on $\ru(G/L)$.

\medskip
We further record that the argment given above in Section~\ref{sec:char} shows that $L\succeq_G H$ is equivalent (for locally compact groups) to the following variant of~\ref{pt:equiv:op} in Theorem~\ref{thm:equiv} above: There is a $G$-equivariant continuous linear map $\alpha\colon L^\infty(G/L)\to L^\infty(G/H)$ with $\alpha(\one_{G/L})=\one_{G/H}$. Arguing as before, one can moreover obtain $\alpha$ to be positive and normalized.

\subsection{Are there counter-examples?}\label{sec:exo}
\begin{flushright}
\begin{minipage}[t]{0.65\linewidth}\itshape\small
There is not the smallest probability that, after having been as obstinate as a mule for two years, she suddenly became amenable\ldots
\begin{flushright}
\upshape\tiny
H. James, \emph{Washington Square},\\
Macmillan, London (1921),  p.~212.
\end{flushright}
\end{minipage}
\end{flushright}
\vspace{5mm}
We have already pointed out that Theorem~\ref{thm:closure} should make it very difficult to find a non-amenable group $H$ appearing as a relatively amenable closed subgroup of a locally compact group $G$. Turning our horses around, we shall now discuss some of the limitations of Theorem~\ref{thm:closure}.

\medskip
In contrast to assertions~\ref{pt:class:ameneq} and~\ref{pt:class:conneq}, neither~\ref{pt:class:td:discrete} nor~\ref{pt:class:td:mai} are likely to admit a converse. Indeed, it would then follow in both cases that \emph{all} locally compact groups belong to~$\cx$.

The reason is as follows. As we have established, it suffices to consider a compactly generated totally disconnected locally compact group $G$. After possibly factoring out a compact kernel, any such group can be written as the quotient of a closed subgroup $\widetilde G$ of the automorphism group of a locally finite regular tree by a normal subgroup $N$ which is free and discrete. (This is explained in~\cite[3.4]{Burger-Monod} or in~\cite[p.~150]{Monod_LN}.) Now $\widetilde G$ is amenable at infinity (this is well-known and follows, for instance, from the general results in~\cite{Lecureux}). As to $N$, it is both discrete and amenable at infinity. Picking up the pieces, a converse to either~\ref{pt:class:td:discrete} or~\ref{pt:class:td:mai} would imply that $G$ is in~$\cx$.

\medskip
We do not know if the class~$\cx$ is closed under taking arbitrary extensions. We claim that this question can be reduced to the following: \itshape
Let $G$ be a semi-direct product $G=U\ltimes N$ with $U$ profinite and $N\in\cx$ totally disconnected; is $G$ in~$\cx$?\upshape

Indeed, if we attempt to apply Lemma~\ref{lem:ext} to $\cx_0=\cx$, its first two stability assumptions are granted by Theorem~\ref{thm:closure}. This leaves us with the third, and thus the general extension problem is reduced to the case where $G/N$ is compact. The same line of reasoning as in the proof of Lemma~\ref{lem:ext} further reduces us to the case where $G$ is totally disconnected. Now $G$ admits an open profinite subgroup $U$ by van Dantzig's theorem, and we can assume $G=U.N$ by an application of Lemma~\ref{lem:UN}. Now $G$ is canonically presented as a quotient of a semi-direct product $G=U\ltimes N$ by the compact kernel $N\cap U$, so that our claim follows using Theorem~\ref{thm:closure}\ref{pt:class:ameneq}.

\medskip
Regarding extensions, we also do not know whether the class of locally compact groups amenable at infinity is closed under extensions (even in the special case where the quotient is amenable); the discrete case is established in~\cite[5.2.6]{Anantharaman-Renault}. Such extensions do, however, belong to~$\cx$ as a consequence of Theorem~\ref{thm:closure}, points~\ref{pt:class:mai} and~\ref{pt:class:td:mai}.

\bigskip
There is an equivalent reformulation of the criterion of Theorem~\ref{thm:equiv}\ref{pt:equiv:op} for relative amenability which highlights the measurability pitfalls that might blur the distinction between relative amenability and usual amenability of $H<G$. The map~$\alpha$ corresponds to a $G$-equivariant element
$$A\in\lw\Big( G/H, (L^\infty(G))^*\Big).$$
At first sight, this seems to mean a weak-* measurable assignment of a mean $A(x)$ on $L^\infty(G)$ for each $x\in G/H$. Now $G$-equivariance would imply that $A(x)$ is fixed by the $x$-conjugate of $H$ in $G$, which implies that this conjugate is amenable (Theorem~\ref{thm:classic}\ref{pt:classic:mean}) and hence $H$ is amenable.

Now of course $A$ is only a function \emph{class}. The fact that $L^\infty(G)$ is non-separable is not of much concern: if we are willing to assume $G$ second countable, a compactness argument reduces us to work in the situation where the mean is defined on a separable $G$-invariant closed subspace $E\se L^\infty(G)$. Moreover, a lifting argument allows us to choose a representative $\widetilde A$ of $A$ with good properties. Summing up, what all this means is that we have a map $\widetilde A\colon G/H\to E^*$ (everywhere defined) such that

\begin{enumerate}[label=(\arabic*)]
\item for all $f\in E$, the map $x\mapsto \langle \widetilde A(x), f\rangle$ is measurable,
\item $\langle \widetilde A(x), \one_G\rangle = 1$ for a.e.~$x$,
\item for any $g\in G$ and any $f\in E$ we have $\langle  \widetilde A(g x), f\rangle = \langle  \widetilde A(x), g f\rangle$ for a.e.~$x$.
\end{enumerate}

In contrast to other proofs above, one cannot continue the argument by applying ergodicity to the new map $x\mapsto g  \widetilde A(g\inv x)$ since the action on $E^*$ is not weak-* measurable unless $G$ is discrete.

\subsection{A structure result}\label{sec:structure}
Finally, we present a structure result for certain amenable actions on locally compact spaces which is related to relative amenability as follows. In an earlier stage of this work, we realised that relative amenability coincides with amenability as soon as $G$ admits an amenable action on a locally compact space $X$ such that the $G$-representation on $\cb(X)$ is continuous (for the sup-norm). There are two obvious examples where this continuity holds:

\smallskip
\begin{enumerate}[label=(\arabic*)]
\item When $X$ is compact: in this case, amenability of the action implies that $G$ is amenable at infinity.
\item When $G$ is discrete. This extends immediately to the case where $G$ is amenable-by-discrete.
\end{enumerate}

\smallskip
The theorem below shows that in fact these two cases are the only ones. In particular, any group $G$ admitting such an action belongs to the class~$\cx$.

\begin{thm}\label{thm:decompose}
Let $G$ be any $\sigma$-compact locally compact group. The following are equivalent:
\begin{enumerate}[label=(\roman*)]
\item $G$ admits an amenable continuous action on a locally compact space $X$ such that the $G$-representation on $\cb(X)$ is continuous.
\item $G$ is either amenable at infinity or amenable-by-discrete.
\end{enumerate}
\end{thm}

This criterion relies on the following characterisation of the continuity of the representation on $\cb(X)$. The reformulation in terms of the Stone--\v{C}ech compactification $\beta X$ is a matter of definitions, whilst (a)$\Longrightarrow$(b) is the core of the statement.

\begin{prop}\label{prop:decompose}
Let $G$ be a locally compact group acting continuously on a locally compact space $X$. We assume that $X$ is $\sigma$-compact and that $G$ has a countable basis of identity neighbourhoods. The following are equivalent:
\begin{enumerate}
\item[(a)] The $G$-representation on $\cb(X)$ is continuous (for the sup-norm).
\item[(a')] The $G$-action on $\beta X$ is continuous.
\item[(b)] There is an open subgroup $O<G$ preserving a compact set $K\se X$ and acting trivially outside $K$.
\end{enumerate}
\end{prop}

\begin{proof}[Proof of Proposition~\ref{prop:decompose}]
(a)$\Longrightarrow$(b). Let $\{U_n\}_{n\in\NN}$ be  a basis of identity neighbourhoods in $G$ and $\{C_n\}_{n\in\NN}$ a sequence of compact subsets of $X$ whose union covers $X$. Suppose~(b) fails. We can assume that each $U_n$ is compact. We construct inductively $g_n\in G$, $x_n\in X$, a compact subset $K_n\se X$ and a continuous function $f_n\colon X\to [0,1]$ as follows, with $g_0$, $x_0$, $K_0$ and $f_0$ arbitrary. Let $n\geq 1$. Define $K_n= C'_{n-1} \cup K_{n-1}\cup \{x_{n-1}, g_{n-1}x_{n-1}\}$, where $C'_{n-1}$ is any compact neighbourhood of $C_{n-1}$. By assumption, the subgroup generated by $U_n$ must move some point outside the compact set $K_{n-1}\cup U_n\inv K_{n-1}$. We can thus choose $x_n$ and $g_n$ such that $x_n\notin K_{n-1}$, $g_n\in U_n$, $g_n x_n\neq x_n$ and $g_n x_n \notin K_{n-1}$. Applying Tietze's extension theorem to the compact set $K_{n-1}\cup\{x_n, g_n x_n\}$ we obtain a continuous function $f_n\colon X\to [0,1]$ which coincides with $f_{n-1}$
  on $K_{n-1}$ and satisfies $f_n(x_n)=0$, $f_n(g_n x_n)=1$.

\smallskip
Since the sequence $K_n$ is increasing and covers $X$, the sequence $f_n$ converges to a function $f\colon X\to [0,1]$. In fact, the convergence is uniform on compact subsets since $K_n$ contains a neighbourhood of $C_j$ for $j<n$; therefore $f$ is continuous. However, $f(x_n)=0$ and $f(g_n x_n)=1$ for all $n\geq 1$ even though $g_n$ converges to the identity, contradicting~(a).

\medskip
(b)$\Longrightarrow$(a'). Since $G$ acts by homeomorphisms on $\beta X$, it suffices to show that the $O$-action on $\beta X$ is continuous. In fact it is enough to consider the action on $\beta X\setminus K$ since this is a neighbourhood of $\beta X\setminus X$. The latter action is trivial since $\beta X\setminus K$ is in the closure of $X\setminus K$ in $\beta X$.

\medskip
(a')$\Longrightarrow$(a). For any compact continuous $G$-space $Z$, the $G$-representation on $\cc(Z)$ is continuous for the sup-norm. Therefore this implication follows from the natural identification $\cb(X)\cong\cc(\beta X)$.
\end{proof}

\begin{proof}[Proof of Theorem~\ref{thm:decompose}]
(ii)$\Longrightarrow$(i). If $G$ is amenable at infinity, then by definition it acts amenably on a compact space $X$; we recall that the action on $\cb(X)=\cc(X)$ is continuous by compactness. If $G$ has an amenable open normal subgroup $A\triangleleft G$, we let $X=G/A$.

\medskip
(i)$\Longrightarrow$(ii). If $X$ is compact, then $G$ is amenable at infinity. We assume henceforth that $X$ is non-compact and proceed to construct an amenable open normal subgroup $A\triangleleft G$.

Since $G$ is $\sigma$-compact, there is by~\cite{KakutaniKodaira} a compact normal subgroup $N\triangleleft G$ such that $G/N$ is second countable. The quotient $X/N$ is a locally compact space with a continuous $G/N$-action. Moreover, the amenability, non-compactness and continuity on $\cb(X/N)$ still all hold --- the latter because of the canonical inclusion $\cb(X/N)\se \cb(X)$. A standard procedure provides a second countable equivariant quotient $X/N\twoheadrightarrow Y$ which still retains all these properties. (This consists in taking $Y$ to be the spectrum of a separable $G$-invariant $C^*$-subalgebra of $\cc_0(X)$ large enough to define the sequence of maps in the definition of topologically amenable actions~\cite{Anantharaman02}; a countable sequence suffices since $G$ is second countable.) Now Proposition~\ref{prop:decompose} provides an open subgroup $O<G/N$ acting trivially outside some compact subset of $Y$. The pre-image $O'$ of $O$ in $G$ is open and we claim that the normal subgroup $A$ of $G$ normally generated by $O'$ is amenable; this will complete the proof.

To this end, is suffices to show that for any finite set $F\se G$, the subgroup of $G$ generated by $\bigcup_{g\in F}g O' g\inv$ is amenable. Since $Y$ is non-compact, the groups $g O g\inv$ have a common fixed point. Now we conclude since the stabiliser of any point in an amenable action is an amenable subgroup.
\end{proof}



\bigskip
\bibliographystyle{amsalpha}
\bibliography{../IsomCAT0}

\end{document}